\documentclass[12pt,reqno]{amsart}

\usepackage{amsfonts,amsmath,amsthm}
\usepackage{amssymb,epsfig}
\usepackage{enumerate} 

\usepackage[text={425pt,650pt},centering]{geometry}
\usepackage{graphicx}
\usepackage{epsfig}
\usepackage{tikz}
\usepackage{caption}
\usepackage{color} 
\usepackage{cite}
\definecolor{vert}{rgb}{0,0.6,0}

\numberwithin{figure}{section}

\theoremstyle{plain}
\newtheorem{thm}{Theorem}[section]

\newtheorem{lem}[thm]{Lemma}

\newtheorem{prop}[thm]{Proposition}
\theoremstyle{remark}
\newtheorem{rem}{\bf{Remark}}
\numberwithin{equation}{section}

\newcommand{\R}{\mathbb{R}}

\newcommand{\T}{\mathbb{T}}

\newcommand{\Lip}{{\rm Lip\,}}

\newcommand{\del}{\delta}
\newcommand{\ep}{\varepsilon}

\newcommand{\sig}{\sigma}

\newcommand{\Del}{\Delta}

\newcommand{\Div}{{\rm div}\,}

\begin{document}
	
	\title[]
	{Convergence Rates for Vanishing Viscosity Approximations of Possibly Degenerate Viscous Hamilton--Jacobi Equations}

	\author[J. Qi]
	{Junkai Qi}
	
	\address[J. Qi]
	{Department of Mathematics, University of Wisconsin--Madison, Madison, WI 53706, USA}
	\email{jqi64@wisc.edu}
	
	
	\begin{abstract}
		We study quantitative convergence rates for vanishing viscosity approximations
		of possibly degenerate viscous Hamilton--Jacobi equations on the flat torus.
		The limiting equation contains a spatially dependent diffusion coefficient
		\(a(x)\ge0\), which is allowed to vanish. Under standard structural assumptions
		on the Hamiltonian, we first prove a pointwise
		convergence rate of order \(O(\ep|\log\ep|)\). We then show that, when the
		error is tested against a smooth probability density, the logarithmic loss can
		be removed and an averaged \(O(\ep)\) rate holds. The proof is based on the
		nonlinear adjoint method, weighted Hessian estimates, and entropy estimates for
		the adjoint density.
	\end{abstract}

	\maketitle

	\section{Introduction}
	\label{sec:introduction}
	
	We study quantitative convergence rates for viscous approximations of degenerate viscous Hamilton--Jacobi equations on the flat
	torus \(\T^n\). Let \(T>0\), and consider
	\begin{equation}\label{eq:intro-deg}
		\begin{cases}
			u_t+H(x,Du)=a(x)\Del u,
			& (x,t)\in \T^n\times(0,T],\\[1mm]
			u(x,0)=g(x),
			& x\in\T^n .
		\end{cases}
	\end{equation}
	Here, \(a\) is nonnegative and may vanish. For \(0<\ep<1\), we introduce the
	viscous approximation
	\begin{equation}\label{eq:intro-eps}
		\begin{cases}
			u^\ep_t+H(x,Du^\ep)=(a(x)+\ep)\Del u^\ep,
			& (x,t)\in \T^n\times(0,T],\\[1mm]
			u^\ep(x,0)=g(x),
			& x\in\T^n .
		\end{cases}
	\end{equation}
	Our purpose is to estimate the rate at which \(u^\ep\) converges, as
	\(\ep\to0\), to the viscosity solution \(u\) of \eqref{eq:intro-deg}.
	
	Throughout the paper, we assume that
	\begin{equation}\label{eq:intro-assumptions}
		\left\{
		\begin{aligned}
			&g\in C^2(\T^n),\qquad a\in C^2(\T^n),\qquad a\ge0,\\
			&H\in C^2(\T^n\times\R^n),\qquad
			D^2_{pp}H(x,p)\ge \theta I_n
			\quad\text{for some }\theta>0 .
		\end{aligned}
		\right.
	\end{equation}
	 We also assume the standard growth condition
	\begin{equation}\label{eq:H-str-assumption}
		|D_xH(x,p)|\le C(1+|p|^2),
		\qquad (x,p)\in \T^n\times\R^n .
	\end{equation}
	Under these assumptions, the Bernstein estimate for viscous
	Hamilton--Jacobi equations gives, for every fixed \(T>0\),
	\begin{equation}\label{eq:intro-lipschitz}
		\|Du^\ep\|_{L^\infty(\T^n\times[0,T])}
		+
		\|u^\ep_t\|_{L^\infty(\T^n\times[0,T])}
		\le C,
	\end{equation}
	where \(C\) is independent of \(\ep\); see Cagnetti, Gomes, Mitake, and
	Tran~\cite{CGMT15}. In particular, all derivatives of \(H\) that appear below
	are evaluated on a fixed compact subset of \(\T^n\times\R^n\). Hence their
	bounds can be absorbed into the constants throughout the paper.
	
	By the stability of viscosity solutions, we have
	\[
	u^\ep\to u
	\qquad\text{uniformly on }\T^n\times[0,T].
	\]
	The main question of this paper is therefore the following: once convergence
	is known, how fast does \(u^\ep\) converge to \(u\)?

	\subsection{Background and motivation}
	
	The vanishing viscosity method is a classical approximation procedure for
	Hamilton--Jacobi equations. For first-order Hamilton--Jacobi equations and
	their viscous approximations, well-posedness, stability, and convergence are
	now standard; see, for instance,
	\cite{Lions1982,CrandallLions1984,Tran21}. Under general Lipschitz
	assumptions, the usual \(L^\infty\) convergence rate is \(O(\sqrt\ep)\). Related convergence-rate questions for state--constraint Hamilton--Jacobi
	equations were studied by Han and Tu, who obtained interior
	\(O(\sqrt\ep)\)-type estimates and one-sided improved rates under additional
	assumptions \cite{HanTu2022}.
	
	A powerful tool in the quantitative study of such problems is the nonlinear
	adjoint method. Evans introduced this method for Hamilton--Jacobi equations
	in \cite{Evans2010}. Tran then used adjoint equations to study convergence
	rates for static Hamilton--Jacobi equations \cite{Tran2011}. Later,
	Cagnetti, Gomes, Mitake, and Tran further developed the method to study
	large-time behavior for degenerate viscous Hamilton--Jacobi equations
	\cite{CGMT15}.
	
	Selection problems form another related direction. In the first-order case,
	Davini, Fathi, Iturriaga, and Zavidovique studied the vanishing-discount
	problem and characterized the selected limit by Mather measures
	\cite{DaviniFathiIturriagaZavidovique2016}. In the degenerate viscous case,
	Mitake and Tran obtained an analogous selection result using the nonlinear
	adjoint method and stochastic Mather measures \cite{MitakeTran2017}. More
	recently, Chen and Zhu considered a new selection problem for degenerate
	viscous Hamilton--Jacobi equations involving a nonlinear discount and a small
	potential perturbation \cite{ChenZhu2026}.
	
	Under additional convexity or regularity assumptions, one can improve the
	standard \(O(\sqrt\ep)\) rate. Camilli, Goffi, and Mendico obtained integral
	estimates and \(L^p\)-type convergence rates via the nonlinear adjoint
	method \cite{CamilliGoffiMendico2024}. Chaintron and Daudin proved the sharp
	\(O(\ep|\log\ep|)\) rate for quadratic Hamilton--Jacobi equations
	\cite{ChaintronDaudin2025}, and later extended this type of estimate to
	uniformly convex Hamiltonians under suitable assumptions
	\cite{ChaintronDaudinUniform2025}. Cirant and Goffi obtained \(L^\infty\)
	convergence rates for convex Hamilton--Jacobi equations, including the
	\(O(\ep|\log\ep|)\) rate under appropriate convexity and regularity
	conditions \cite{CirantGoffi2025}. See also the work of Wang and Zhang on
	the vanishing viscosity limit with nearly optimal discount
	\cite{WangZhangDiscount2026}.
	
	Related convergence-rate problems also arise in periodic homogenization.
	Qian, Sprekeler, Tran, and Yu proved the optimal \(O(\sqrt\ep)\) convergence
	rate for periodic homogenization of viscous Hamilton--Jacobi equations under
	general assumptions \cite{QianSprekelerTranYu2024}. More recently, Liu,
	Tran, and Yu proved a sharp global \(O(\ep|\log\ep|)\)-type estimate for
	periodic homogenization of viscous quadratic Hamilton--Jacobi equations,
	together with an \(O(\ep)\) improvement at almost every point under an
	additional semiconcavity assumption \cite{LiuTranYu2026}.
	
	The purpose of the present paper is to develop further the nonlinear adjoint method
	to study finer properties of solutions of degenerate viscous Hamilton--Jacobi equations with spatially dependent
	diffusion. The main difficulty is caused by the term \(a(x)\Delta u^\ep\).
	In the presence of this coefficient, the classical Bernstein argument no
	longer applies directly: after differentiating the equation, one obtains
	terms involving derivatives of \(a\) coupled with higher-order derivatives
	of \(u^\ep\). Such terms do not appear in the first-order case, and they
	cannot be controlled by the standard estimates alone. To handle this
	difficulty, we combine the nonlinear adjoint method with entropy estimates
	for the adjoint density and weighted second-derivative estimates for
	\(u^\ep\). This yields the logarithmic convergence rate in the degenerate
	setting and also gives an averaged \(O(\ep)\)-type estimate for smooth
	terminal densities.
	
	\subsection{Main results}
	
	We now state the two main results of this paper. The first one is a pointwise
	convergence estimate with the logarithmic rate \(O(\ep|\log\ep|)\).
	
	\begin{thm}\label{thm:intro-log-rate}
		Assume \eqref{eq:intro-assumptions} and \eqref{eq:H-str-assumption}. Let
		\(u^\ep\) be the solution of \eqref{eq:intro-eps}, and let \(u\) be the
		viscosity solution of \eqref{eq:intro-deg}. Then there exists a constant
		\(C>0\), independent of \(\ep\), such that
		\vspace{1mm}
		\begin{equation}\label{eq:intro-log-rate}
			\|u^\ep-u\|_{L^\infty(\T^n\times[0,T])}
			\le
			C\ep\bigl(1+|\log\ep|\bigr).
		\end{equation}
	\end{thm}
	\vspace{1mm}
	The logarithmic factor can be removed if the error is averaged against a
	smooth probability density. More precisely, let \(r\in C^\infty(\T^n)\)
	satisfy
	\begin{equation}\label{eq:intro-r-density}
		r\ge0,
		\qquad
		\int_{\T^n}r(x)\,dx=1 .
	\end{equation}
	
	\begin{thm}\label{thm:intro-avg-rate}
		Assume \eqref{eq:intro-assumptions} and \eqref{eq:H-str-assumption}. Let \(u^\ep\) be the solution of
		\eqref{eq:intro-eps}, and let \(u\) be the viscosity solution of
		\eqref{eq:intro-deg}. Then there exists a constant \(C>0\), independent of
		\(\ep\) and \(r\), such that
		\vspace{0.5mm}
		\begin{equation}\label{eq:intro-avg-rate}
			\left|
			\int_{\T^n}
			\bigl(u^\ep(x,T)-u(x,T)\bigr)r(x)\,dx
			\right|
			\le
			C\ep
			\left(
			1+\|Dr\|_{L^1(\T^n)}
			+\int_{\T^n}r|\log r|\,dx
			\right)^{1/2}.
		\end{equation}
	\vspace{0.5mm}
		Here \(r|\log r|\) is understood as \(0\) on the set \(\{r=0\}\).
	\end{thm}
	
	To the best of our knowledge, both Theorems~\ref{thm:intro-log-rate}
	and~\ref{thm:intro-avg-rate} are new in the literature for degenerate viscous
	Hamilton--Jacobi equations.
	
	\subsection{Idea of the proof}
	
	We briefly describe the mechanism of the proof. Differentiating
	\eqref{eq:intro-eps} with respect to \(\ep\), we obtain a linearized equation
	for \(w^\ep=\partial_\ep u^\ep\), whose source term is \(\Del u^\ep\). The
	nonlinear adjoint method then reduces the convergence-rate problem to
	weighted estimates for \(\Del u^\ep\).
	
	The basic adjoint estimate only gives the standard \(O(\sqrt\ep)\) rate. To
	obtain the logarithmic rate, we develop an argument in the spirit of Cirant
	and Goffi \cite{CirantGoffi2025}, but adapted to the degenerate viscous
	setting. The main new difficulty is the spatially dependent diffusion
	\(a(x)\), which produces additional \(Da\)-terms after differentiation. These
	terms are handled by combining the structure inequality
	\[
	|Da|^2\le C a
	\]
	with entropy and Fisher-information estimates for the adjoint density. This
	yields the key weighted second-derivative estimate and leads to the
	\(O(\ep|\log\ep|)\) pointwise convergence rate.
	
	For the averaged estimate, we replace the terminal Dirac mass in the adjoint
	equation by a smooth probability density \(r\). This removes the endpoint
	singularity of the adjoint density. As a result, the logarithmic loss
	disappears, and we obtain an averaged \(O(\ep)\) estimate.
	
	The paper is organized as follows. In
	Section~\ref{sec:degenerate-log-rate}, we prove
	Theorem~\ref{thm:intro-log-rate}. In Section~\ref{sec:averaged-estimate}, we
	prove Theorem~\ref{thm:intro-avg-rate}.
	
	\section{Pointwise estimates in the degenerate case}
	\label{sec:degenerate-log-rate}
	
	In this section we prove Theorem~\ref{thm:intro-log-rate}. Throughout this
	section we write
	\[
	A^\ep(x):=a(x)+\ep,
	\qquad
	b^\ep(x,t):=D_pH(x,Du^\ep(x,t)).
	\]
	We use the notation
	\begin{equation}\label{eq:deg-L-eps}
		L^\ep f:=f_t+b^\ep\cdot Df-A^\ep\Del f .
	\end{equation}
	
	For \(0<\tau\le T\) and \(x_0\in\T^n\), let \(\sig^\ep\) solve the adjoint
	equation
	\begin{equation}\label{eq:deg-adjoint-dirac}
		\begin{cases}
			-\sig^\ep_t-\Div(b^\ep\sig^\ep)-\Del(A^\ep\sig^\ep)=0,
			& (x,t)\in \T^n\times(0,\tau),\\[1mm]
			\sig^\ep(\cdot,\tau)=\delta_{x_0},
			& \text{in the sense of measures on } \T^n .
		\end{cases}
	\end{equation}
	Here, \(\delta_{x_0}\) is the Dirac measure at \(x_0\). By the maximum
	principle and the conservation of mass, we have
	\begin{equation}\label{eq:deg-sigma-basic}
		\sig^\ep\ge0,
		\qquad
		\int_{\T^n}\sig^\ep(x,t)\,dx=1
		\quad\text{for }0\le t<\tau .
	\end{equation}
	Moreover, the following duality identity holds:
	\begin{equation}\label{eq:deg-duality-identity}
		f(x_0,\tau)
		-
		\int_{\T^n}f(x,0)\sig^\ep(x,0)\,dx
		=
		\int_0^\tau\int_{\T^n}L^\ep f\,\sig^\ep\,dx\,dt .
	\end{equation}
	
	We shall use the following elementary consequence of the degeneracy condition
	\(a\ge0\).
	
	\begin{lem}\label{lem:Da-structure}
		There exists a constant \(C>0\) such that
		\begin{equation}\label{eq:Da-structure}
			|Da(x)|^2\le C a(x)
			\qquad\text{for all }x\in\T^n .
		\end{equation}
	\end{lem}
	
	\begin{proof}
		Since \(a\ge0\) and \(a\in C^2(\T^n)\), it is standard that \(a^{1/2}\in
		\Lip(\T^n)\). Hence, at every point where \(a(x)>0\),
		\[
		D(a^{1/2})(x)=\frac{Da(x)}{2a^{1/2}(x)}.
		\]
		Therefore
		\[
		|Da(x)|^2
		=
		4a(x)|D(a^{1/2})(x)|^2
		\le C a(x)
		\qquad\text{on }\{a>0\}.
		\]
		The estimate is trivial on \(\{a=0\}\), since every zero point of \(a\) is a
		minimum point and hence \(Da=0\) there.
	\end{proof}
	
	We first prove the standard estimate. It will later be used to control the
	short time intervals near the endpoints.
	
	\begin{thm}\label{thm:deg-sqrt-rate}
		There exists a constant \(C>0\), independent of \(\ep\), such that
		\begin{equation}\label{eq:deg-sqrt-rate}
			\|u^\ep-u\|_{L^\infty(\T^n\times[0,T])}
			\le C\sqrt\ep .
		\end{equation}
	\end{thm}
	
	The proof relies on the following weighted estimate for \(D^2u^\ep\).
	
	\begin{lem}\label{lem:deg-basic-D2}
		Let \(\sig^\ep\) solve \eqref{eq:deg-adjoint-dirac}. Then
		\begin{equation}\label{eq:deg-basic-D2}
			\int_0^\tau\int_{\T^n}
			A^\ep |D^2u^\ep|^2\sig^\ep\,dx\,dt
			\le C .
		\end{equation}
	\end{lem}
	
	\begin{proof}
		Differentiate \eqref{eq:intro-eps} with respect to \(x_i\). We obtain
		\begin{equation}\label{eq:deg-xi-eq}
			u^\ep_{x_i t}
			+
			H_{x_i}(x,Du^\ep)
			+
			H_{p_j}(x,Du^\ep)u^\ep_{x_i x_j}
			=
			A^\ep\Del u^\ep_{x_i}
			+
			a_{x_i}\Del u^\ep .
		\end{equation}
		Multiplying by \(u^\ep_{x_i}\), summing over \(i\), and using
		\[
		\sum_{i=1}^n u^\ep_{x_i}\Del u^\ep_{x_i}
		=
		\Del\left(\frac12|Du^\ep|^2\right)-|D^2u^\ep|^2,
		\]
		we get
		\begin{equation}\label{eq:deg-gradient-energy-identity}
			L^\ep\left(\frac12|Du^\ep|^2\right)
			+
			A^\ep |D^2u^\ep|^2
			=
			-D_xH(x,Du^\ep)\cdot Du^\ep
			+
			(Da\cdot Du^\ep)\Del u^\ep .
		\end{equation}
		Since \(A^\ep=a+\ep\) and \(|Da|^2\le Ca\), we have
		\[
		\frac{|Da|^2}{A^\ep}\le C.
		\]
		Then, by \eqref{eq:intro-lipschitz} and Young's inequality,
		\[	\left|(Da\cdot Du^\ep)\Del u^\ep\right|
		\le
		\frac12 A^\ep |D^2u^\ep|^2+C .\]
		Therefore
		\begin{equation}\label{eq:deg-gradient-energy}
			L^\ep\left(\frac12|Du^\ep|^2\right)
			+
			\frac12 A^\ep |D^2u^\ep|^2
			\le C .
		\end{equation}
		Pairing this inequality with \(\sig^\ep\) and using
		\eqref{eq:deg-duality-identity}, we find
\[
\begin{gathered}
	\frac12
	\int_0^\tau\int_{\T^n}
	A^\ep |D^2u^\ep|^2\sig^\ep\,dx\,dt \\
	\le
	C\tau
	+
	\frac12\int_{\T^n}
	|Du^\ep(x,0)|^2\sig^\ep(x,0)\,dx
	-
	\frac12|Du^\ep(x_0,\tau)|^2
	\le C .
\end{gathered}
\]
		The lemma follows.
	\end{proof}
	
	\begin{proof}[Proof of Theorem~\ref{thm:deg-sqrt-rate}]
		Let
		\[
		w^\ep:=\partial_\ep u^\ep .
		\]
		Differentiating \eqref{eq:intro-eps} with respect to \(\ep\), we obtain
		\begin{equation}\label{eq:deg-weps-eq}
			\begin{cases}
				L^\ep w^\ep=\Del u^\ep,
				& (x,t)\in\T^n\times(0,\tau),\\[1mm]
				w^\ep(x,0)=0,
				& x\in\T^n .
			\end{cases}
		\end{equation}
		By \eqref{eq:deg-duality-identity},
		\begin{equation}\label{eq:deg-weps-adjoint}
			w^\ep(x_0,\tau)
			=
			\int_0^\tau\int_{\T^n}
			\Del u^\ep\,\sig^\ep\,dx\,dt .
		\end{equation}
		Hence, by Lemma~\ref{lem:deg-basic-D2}, \(A^\ep\ge \ep\), and
		\eqref{eq:deg-sigma-basic},
		\begin{equation}\label{eq:deg-weps-sqrt-bound}
			\begin{gathered}
				|w^\ep(x_0,\tau)|
				\le C\int_0^\tau\int_{\T^n}|D^2u^\ep|\sig^\ep\,dx\,dt  \\
				\le C
				\left(
				\int_0^\tau\int_{\T^n}
				A^\ep |D^2u^\ep|^2\sig^\ep\,dx\,dt
				\right)^{1/2}
				\left(
				\int_0^\tau\int_{\T^n}
				\frac{\sig^\ep}{A^\ep}\,dx\,dt
				\right)^{1/2} \le \frac{C}{\sqrt{\ep}} .
			\end{gathered}
		\end{equation}
	    \vspace{0.5mm}
		Indeed,
		\[
		\int_0^\tau\int_{\T^n}
		\frac{\sig^\ep}{A^\ep}\,dx\,dt
		\le
		\frac1\ep\int_0^\tau\int_{\T^n}\sig^\ep\,dx\,dt
		\le \frac{C}{\ep}.
		\]
		Since \(x_0\) and \(\tau\) are arbitrary,
		\begin{equation}\label{eq:deg-weps-sqrt}
			\|\partial_\ep u^\ep\|_{L^\infty(\T^n\times[0,T])}
			\le
			\frac{C}{\sqrt\ep}.
		\end{equation}
		For \(0<\eta<\ep\),
		\[
		u^\ep-u^\eta
		=
		\int_\eta^\ep \partial_su^s\,ds .
		\]
		Thus
		\[
		\|u^\ep-u^\eta\|_{L^\infty(\T^n\times[0,T])}
		\le
		\int_\eta^\ep \frac{C}{\sqrt s}\,ds
		\le
		C\sqrt\ep .
		\]
		Letting \(\eta\to0\) and using the uniform convergence \(u^\eta\to u\), we
		obtain \eqref{eq:deg-sqrt-rate}.
	\end{proof}
	
	We now derive the entropy estimates for the adjoint density.
	
	\begin{lem}\label{lem:deg-entropy}
		Let \(0<\ep<1\), \(4\ep<\tau\le T\), and let \(\sig^\ep\) solve
		\eqref{eq:deg-adjoint-dirac}. Then there exist
		\[
		t_1\in[\ep,2\ep],
		\qquad
		t_2\in[\tau-2\ep,\tau-\ep],
		\]
		such that
		\begin{equation}\label{eq:deg-Dsigma-endpoint}
			\int_{\T^n}|D\sig^\ep(x,t_i)|\,dx
			\le
			\frac{C}{\ep}\bigl(1+|\log\ep|\bigr)^{1/2},
			\qquad i=1,2.
		\end{equation}
		Moreover, with
		\[
		\chi_\tau(t):=t(\tau-t),
		\]
		for every \(\del>0\),
		\begin{equation}\label{eq:deg-weighted-fisher}
			\begin{aligned}
				\int_{t_1}^{t_2}\int_{\T^n}
				A^\ep\chi_\tau(t)\frac{|D\sig^\ep|^2}{\sig^\ep}\,dx\,dt
				&\le
				\del
				\int_{t_1}^{t_2}\int_{\T^n}
				\chi_\tau(t)|D^2u^\ep|^2\sig^\ep\,dx\,dt        \\
				&\quad
				+
				C_\del\bigl(1+|\log\ep|\bigr).
			\end{aligned}
		\end{equation}
	\end{lem}
	
	\begin{proof}
		Introduce the time-reversed density
		\[
		\rho^\ep(x,s):=\sig^\ep(x,\tau-s),
		\qquad 0<s<\tau .
		\]
		Then \(\rho^\ep\) solves
		\begin{equation}\label{eq:deg-rho-reversed}
			\rho^\ep_s-\Del(A^\ep\rho^\ep)+\Div(\widehat b^\ep\rho^\ep)=0,
			\qquad
			\rho^\ep(\cdot,0)=\delta_{x_0},
		\end{equation}
		where
		\[
		\widehat b^\ep(x,s):=-b^\ep(x,\tau-s).
		\]
		The coefficients \(A^\ep\) and \(\widehat b^\ep\) are uniformly bounded, and
		the diffusion matrix \((a+\ep)I_n\) has ellipticity constant \(\lambda_0=\ep\).
		
		We apply \cite[Corollary~7.2.3]{BKRS15} to
		\eqref{eq:deg-rho-reversed}, with diffusion matrix \((a+\ep)I_n\), drift
		\(\widehat b^\ep\), zero potential term, and \(\lambda_0=\ep\). Thus, for
		any \(\nu>(n+2)/2\),
		\begin{equation}\label{eq:deg-rho-upper}
			\rho^\ep(x,s)
			\le
			C_{\nu}
			\left(1+\frac1\ep\right)^\nu
			s^{-\frac n2}
			\left(
			1+\frac{s^{2\nu}}{\ep^\nu}
			\right),
			\qquad 0<s<\tau .
		\end{equation}
		Consequently,
		\begin{equation}\label{eq:deg-log-rho-upper}
			\log\rho^\ep(x,s)
			\le
			C\bigl(1+|\log\ep|+|\log s|\bigr),
			\qquad 0<s<\tau .
		\end{equation}
		Since \(\int_{\T^n}\rho^\ep(x,s)\,dx=1\), and since
		\(-r\log r\le C\) for \(0\le r\le1\), we obtain
		\begin{equation}\label{eq:deg-entropy-abs-bound}
			\int_{\T^n}\sig^\ep(x,t)|\log\sig^\ep(x,t)|\,dx
			\le
			C\bigl(1+|\log\ep|+|\log(\tau-t)|\bigr),
			\qquad 0\le t<\tau .
		\end{equation}
		
		Set
		\[
		E(t):=\int_{\T^n}\sig^\ep(x,t)\log\sig^\ep(x,t)\,dx .
		\]
		Testing the adjoint equation by \(1+\log\sig^\ep\), we get
		\begin{equation}\label{eq:deg-entropy-id}
			E'(t)
			=
			\int_{\T^n}
			A^\ep\frac{|D\sig^\ep|^2}{\sig^\ep}\,dx
			+
			\int_{\T^n}
			(b^\ep+Da)\cdot D\sig^\ep\,dx .
		\end{equation}
		By Young's inequality, 
		\[
		|(b^\ep+Da)\cdot D\sig^\ep|\le \frac{\ep}{2}\frac{|D\sig^\ep|^2}{\sig^\ep}+\frac{C}{\ep}|b^\ep+Da|^2\sig^\ep,
		\quad A^\ep\ge\ep.
		\] 
		Using the boundedness of \(b^\ep\) and
		\(Da\), we have
		\begin{equation}\label{eq:deg-fisher-pointwise}
			\frac{\ep}{2}
			\int_{\T^n}
			\frac{|D\sig^\ep|^2}{\sig^\ep}\,dx
			\le
			E'(t)+\frac{C}{\ep}.
		\end{equation}
		Integrating over \([\ep,2\ep]\), using
		\eqref{eq:deg-entropy-abs-bound} and the lower bound \(E(t)\ge -C\), gives
		\[
		\int_\ep^{2\ep}\int_{\T^n}
		\frac{|D\sig^\ep|^2}{\sig^\ep}\,dx\,dt
		\le
		\frac{C}{\ep}\bigl(1+|\log\ep|\bigr).
		\]
		Hence there exists \(t_1\in[\ep,2\ep]\) such that
		\[
		\int_{\T^n}
		\frac{|D\sig^\ep(x,t_1)|^2}{\sig^\ep(x,t_1)}\,dx
		\le
		\frac{C}{\ep^2}\bigl(1+|\log\ep|\bigr).
		\]
		By H\"older's inequality and conservation of mass,
		\[
		\int_{\T^n}|D\sig^\ep(x,t_1)|\,dx
		\le
		\frac{C}{\ep}\bigl(1+|\log\ep|\bigr)^{1/2}.
		\]
		The same argument on \([\tau-2\ep,\tau-\ep]\) gives the estimate at some
		\(t_2\in[\tau-2\ep,\tau-\ep]\). This proves
		\eqref{eq:deg-Dsigma-endpoint}.
		
		It remains to prove \eqref{eq:deg-weighted-fisher}. Integrating by parts in
		the last term of \eqref{eq:deg-entropy-id}, we obtain
		\begin{equation}\label{eq:deg-entropy-expanded}
			E'(t)
			=
			\int_{\T^n}
			A^\ep\frac{|D\sig^\ep|^2}{\sig^\ep}\,dx
			-
			\int_{\T^n}
			(\Div b^\ep+\Del a)\sig^\ep\,dx .
		\end{equation}
		Since
		\[
		\Div b^\ep
		=
		H_{x_i p_i}(x,Du^\ep)
		+
		H_{p_i p_j}(x,Du^\ep)u^\ep_{x_i x_j},
		\]
		we have, for every \(\del>0\),
		\begin{equation}\label{eq:deg-divb-bound}
			\left|
			\int_{\T^n}
			(\Div b^\ep+\Del a)\sig^\ep\,dx
			\right|
			\le
			\del
			\int_{\T^n}|D^2u^\ep|^2\sig^\ep\,dx
			+
			C_\del .
		\end{equation}
		Multiplying \eqref{eq:deg-entropy-expanded} by \(\chi_\tau(t)\), integrating
		from \(t_1\) to \(t_2\), and integrating by parts in time, we find
		\[
		\begin{aligned}
			\int_{t_1}^{t_2}\int_{\T^n}
			A^\ep\chi_\tau
			\frac{|D\sig^\ep|^2}{\sig^\ep}\,dx\,dt
			&=
			[\chi_\tau(t)E(t)]_{t_1}^{t_2}
			-
			\int_{t_1}^{t_2}\chi_\tau'(t)E(t)\,dt        \\
			&\quad
			+
			\int_{t_1}^{t_2}\int_{\T^n}
			\chi_\tau(t)(\Div b^\ep+\Del a)\sig^\ep\,dx\,dt .
		\end{aligned}
		\]
		By \eqref{eq:deg-entropy-abs-bound} and the choice of \(t_1,t_2\),
		\[
		\left|
		[\chi_\tau(t)E(t)]_{t_1}^{t_2}
		-
		\int_{t_1}^{t_2}\chi_\tau'(t)E(t)\,dt
		\right|
		\le
		C\bigl(1+|\log\ep|\bigr).
		\]
		Combining this estimate with \eqref{eq:deg-divb-bound} gives
		\eqref{eq:deg-weighted-fisher}.
	\end{proof}
	
	We next prove the weighted second-derivative estimate which is the main
	estimate in this section.
	
	\begin{prop}\label{prop:deg-weighted-D2}
		Let \(4\ep<\tau\le T\), let \(\sig^\ep\) solve
		\eqref{eq:deg-adjoint-dirac}, and let \(t_1,t_2\) be chosen as in
		Lemma~\ref{lem:deg-entropy}. Then
		\begin{equation}\label{eq:deg-weighted-D2}
			\int_{t_1}^{t_2}\int_{\T^n}
			\chi_\tau(t)|D^2u^\ep|^2\sig^\ep\,dx\,dt
			\le
			C\bigl(1+|\log\ep|\bigr).
		\end{equation}
	\end{prop}
	
	\begin{proof}
		Set
		\[
		\phi^\ep:=\Del u^\ep .
		\]
		Applying \(\Del\) to \eqref{eq:intro-eps}, we get
		\[
		\begin{gathered}
			\phi^\ep_t
			+
			b^\ep\cdot D\phi^\ep
			-
			A^\ep\Del\phi^\ep
			-
			2Da\cdot D\phi^\ep                                      \\
			\quad
			+
			H_{p_jp_k}(x,Du^\ep)u^\ep_{x_i x_j}u^\ep_{x_i x_k}
			+
			2H_{x_i p_j}(x,Du^\ep)u^\ep_{x_i x_j}
			+
			H_{x_i x_i}(x,Du^\ep)
			-
			(\Del a)\phi^\ep
			=0 .
		\end{gathered}
		\]
		Here and below repeated indices are summed. By
		\eqref{eq:intro-assumptions},
		\[
		H_{p_jp_k}(x,Du^\ep)u^\ep_{x_i x_j}u^\ep_{x_i x_k}
		\ge
		\theta |D^2u^\ep|^2 .
		\]
		The remaining terms involving \(H_{xp}\), \(H_{xx}\), and \(\Del a\) are
		controlled by Young's inequality and the boundedness of \(Du^\ep\). Hence
		\begin{equation}\label{eq:deg-bernstein-short}
			L^\ep\phi^\ep
			-
			2Da\cdot D\phi^\ep
			+
			\frac{\theta}{2}|D^2u^\ep|^2
			\le C .
		\end{equation}
		Multiplying by \(\chi_\tau(t)\), we obtain
		\[
		L^\ep(\chi_\tau\phi^\ep)
		+
		\frac{\theta}{2}\chi_\tau |D^2u^\ep|^2
		\le
		C
		+
		\chi_\tau'\phi^\ep
		+
		2\chi_\tau Da\cdot D\phi^\ep .
		\]
		Pairing this inequality with \(\sig^\ep\) on \([t_1,t_2]\), we get
		\begin{equation}\label{eq:deg-weighted-D2-start}
			\begin{aligned}
				\frac{\theta}{2}
				\int_{t_1}^{t_2}\int_{\T^n}
				\chi_\tau |D^2u^\ep|^2\sig^\ep\,dx\,dt
				&\le
				\int_{\T^n}\chi_\tau(t_1)\phi^\ep(x,t_1)\sig^\ep(x,t_1)\,dx       \\
				&\quad
				-
				\int_{\T^n}\chi_\tau(t_2)\phi^\ep(x,t_2)\sig^\ep(x,t_2)\,dx       \\
				&\quad
				+
				\int_{t_1}^{t_2}\int_{\T^n}
				\chi_\tau'\phi^\ep\sig^\ep\,dx\,dt                               \\
				&\quad
				+
				2\int_{t_1}^{t_2}\int_{\T^n}
				\chi_\tau Da\cdot D\phi^\ep\,\sig^\ep\,dx\,dt
				+
				C .
			\end{aligned}
		\end{equation}
		
		For the boundary terms, integration by parts gives
		\[
		\int_{\T^n}\phi^\ep(x,t_i)\sig^\ep(x,t_i)\,dx
		=
		-\int_{\T^n}Du^\ep(x,t_i)\cdot D\sig^\ep(x,t_i)\,dx .
		\]
		Since \(\chi_\tau(t_i)\le C\ep\), \eqref{eq:intro-lipschitz} and
		\eqref{eq:deg-Dsigma-endpoint} imply
		\begin{equation}\label{eq:deg-boundary-D2}
			\left|
			\int_{\T^n}\chi_\tau(t_i)\phi^\ep(x,t_i)\sig^\ep(x,t_i)\,dx
			\right|
			\le
			C\bigl(1+|\log\ep|\bigr),
			\qquad i=1,2 .
		\end{equation}
		
		For the term containing \(\chi_\tau'\), Young's inequality gives
		\begin{equation}\label{eq:deg-chip-term}
			\begin{aligned}
				\left|
				\int_{t_1}^{t_2}\int_{\T^n}
				\chi_\tau'\phi^\ep\sig^\ep\,dx\,dt
				\right|
				&\le
				\del
				\int_{t_1}^{t_2}\int_{\T^n}
				\chi_\tau |D^2u^\ep|^2\sig^\ep\,dx\,dt       \\
				&\quad
				+
				C_\del
				\int_{t_1}^{t_2}
				\frac{|\chi_\tau'(t)|^2}{\chi_\tau(t)}\,dt .
			\end{aligned}
		\end{equation}
		Because \(t_1\in[\ep,2\ep]\) and \(t_2\in[\tau-2\ep,\tau-\ep]\),
		\[
		\int_{t_1}^{t_2}
		\frac{|\chi_\tau'(t)|^2}{\chi_\tau(t)}\,dt
		\le
		C\bigl(1+|\log\ep|\bigr).
		\]
		Thus
		\begin{equation}\label{eq:deg-chip-term-final}
			\left|
			\int_{t_1}^{t_2}\int_{\T^n}
			\chi_\tau'\phi^\ep\sig^\ep\,dx\,dt
			\right|
			\le
			\del
			\int_{t_1}^{t_2}\int_{\T^n}
			\chi_\tau |D^2u^\ep|^2\sig^\ep\,dx\,dt
			+
			C_\del\bigl(1+|\log\ep|\bigr).
		\end{equation}
		
		It remains to estimate the term involving \(Da\). Let
		\[
		J:=
		\int_{t_1}^{t_2}\int_{\T^n}
		\chi_\tau Da\cdot D\phi^\ep\,\sig^\ep\,dx\,dt .
		\]
		Integrating by parts in space,
		\[
		J
		=
		-\int_{t_1}^{t_2}\int_{\T^n}
		\chi_\tau \phi^\ep(\Del a)\sig^\ep\,dx\,dt
		-
		\int_{t_1}^{t_2}\int_{\T^n}
		\chi_\tau \phi^\ep Da\cdot D\sig^\ep\,dx\,dt .
		\]
		The first term is bounded by
		\begin{equation}\label{eq:deg-Da-first}
			\left|
			\int_{t_1}^{t_2}\int_{\T^n}
			\chi_\tau \phi^\ep(\Del a)\sig^\ep\,dx\,dt
			\right|
			\le
			\del
			\int_{t_1}^{t_2}\int_{\T^n}
			\chi_\tau |D^2u^\ep|^2\sig^\ep\,dx\,dt
			+
			C_\del .
		\end{equation}
		For the second term, Lemma~\ref{lem:Da-structure} and \(a\le A^\ep\) give
		\[
		|Da|^2\le C A^\ep .
		\]
		Therefore, by Young's inequality,
		\[
		\begin{aligned}
			\left|
			\int_{t_1}^{t_2}\int_{\T^n}
			\chi_\tau \phi^\ep Da\cdot D\sig^\ep\,dx\,dt
			\right|
			&\le
			\del
			\int_{t_1}^{t_2}\int_{\T^n}
			\chi_\tau |\phi^\ep|^2\sig^\ep\,dx\,dt                         \\
			&\quad
			+
			C_\del
			\int_{t_1}^{t_2}\int_{\T^n}
			A^\ep\chi_\tau
			\frac{|D\sig^\ep|^2}{\sig^\ep}\,dx\,dt .
		\end{aligned}
		\]
		Since \(|\phi^\ep|^2\le n|D^2u^\ep|^2\), Lemma~\ref{lem:deg-entropy} implies
		\begin{equation}\label{eq:deg-Da-second}
			\left|
			\int_{t_1}^{t_2}\int_{\T^n}
			\chi_\tau \phi^\ep Da\cdot D\sig^\ep\,dx\,dt
			\right|
			\le
			\del
			\int_{t_1}^{t_2}\int_{\T^n}
			\chi_\tau |D^2u^\ep|^2\sig^\ep\,dx\,dt
			+
			C_\del\bigl(1+|\log\ep|\bigr).
		\end{equation}
		Here and below the small constants in Young's inequality and in the Fisher
		information estimate are chosen successively; we keep denoting the resulting
		small parameter by \(\del\). Combining \eqref{eq:deg-Da-first} and
		\eqref{eq:deg-Da-second}, we obtain
		\begin{equation}\label{eq:deg-Da-term-final}
			|J|
			\le
			\del
			\int_{t_1}^{t_2}\int_{\T^n}
			\chi_\tau |D^2u^\ep|^2\sig^\ep\,dx\,dt
			+
			C_\del\bigl(1+|\log\ep|\bigr).
		\end{equation}
		
		Substituting \eqref{eq:deg-boundary-D2},
		\eqref{eq:deg-chip-term-final}, and \eqref{eq:deg-Da-term-final} into
		\eqref{eq:deg-weighted-D2-start}, and choosing \(\del>0\) sufficiently small,
		we obtain \eqref{eq:deg-weighted-D2}.
	\end{proof}

	\begin{rem}
		When \(a\equiv 0\), the corresponding convex case was treated by
		Cirant and Goffi \cite{CirantGoffi2025}. The new difficulty here is caused
		by the terms involving \(Da\), which contain
		\(D\phi^\ep=D\Delta u^\ep\). After integration by parts, these terms produce
		derivatives of the adjoint density, and hence must be controlled by the
		Fisher-information estimate.
	\end{rem}

	We can now prove the main estimate.
	
	\begin{proof}[Proof of Theorem~\ref{thm:intro-log-rate}]
		It suffices to estimate \(\partial_\ep u^\ep\). Fix \(x_0\in\T^n\) and
		\(0<\tau\le T\), and let \(\sig^\ep\) solve
		\eqref{eq:deg-adjoint-dirac}. By \eqref{eq:deg-weps-adjoint},
		\begin{equation}\label{eq:deg-main-weps}
			\partial_\ep u^\ep(x_0,\tau)
			=
			\int_0^\tau\int_{\T^n}
			\Del u^\ep\,\sig^\ep\,dx\,dt .
		\end{equation}
		
		If \(\tau\le4\ep\), then Lemma~\ref{lem:deg-basic-D2} gives
		\[
		\begin{gathered}
			|\partial_\ep u^\ep(x_0,\tau)|
			\le
			C
			\int_0^\tau\int_{\T^n}
			|D^2u^\ep|\sig^\ep\,dx\,dt                                    \\
			\le
			C
			\left(
			\int_0^\tau\int_{\T^n}
			A^\ep |D^2u^\ep|^2\sig^\ep\,dx\,dt
			\right)^{1/2}
			\left(
			\int_0^\tau\int_{\T^n}
			\frac{\sig^\ep}{A^\ep}\,dx\,dt
			\right)^{1/2}                                                  
			\le
			C .
		\end{gathered}
		\]
		We may therefore assume \(4\ep<\tau\le T\).
		
		Let \(t_1,t_2\) be chosen as in Lemma~\ref{lem:deg-entropy}. We split
		\[
		[0,\tau]=[0,t_1]\cup[t_1,t_2]\cup[t_2,\tau].
		\]
		On the two endpoint intervals, Lemma~\ref{lem:deg-basic-D2} yields
		\[
		\begin{aligned}
			\int_0^{t_1}\int_{\T^n}
			|D^2u^\ep|\sig^\ep\,dx\,dt
			&\le
			C
			\left(
			\int_0^{t_1}\int_{\T^n}
			\frac{\sig^\ep}{A^\ep}\,dx\,dt
			\right)^{1/2}
			\le C,                                                        \\
			\int_{t_2}^{\tau}\int_{\T^n}
			|D^2u^\ep|\sig^\ep\,dx\,dt
			&\le C .
		\end{aligned}
		\]
		For the middle interval, Proposition~\ref{prop:deg-weighted-D2} gives
		\[
		\begin{aligned}
			\int_{t_1}^{t_2}\int_{\T^n}
			|D^2u^\ep|\sig^\ep\,dx\,dt
			&\le
			\left(
			\int_{t_1}^{t_2}\int_{\T^n}
			\chi_\tau |D^2u^\ep|^2\sig^\ep\,dx\,dt
			\right)^{1/2}
			\left(
			\int_{t_1}^{t_2}\int_{\T^n}
			\frac{\sig^\ep}{\chi_\tau}\,dx\,dt
			\right)^{1/2}                                                  \\
			&\le
			C\bigl(1+|\log\ep|\bigr).
		\end{aligned}
		\]
		Here we used conservation of mass and
		\[
		\int_{t_1}^{t_2}\frac{dt}{t(\tau-t)}
		\le
		C\bigl(1+|\log\ep|\bigr).
		\]
		Combining the three estimates with \eqref{eq:deg-main-weps}, we obtain
		\[
		|\partial_\ep u^\ep(x_0,\tau)|
		\le
		C\bigl(1+|\log\ep|\bigr).
		\]
		Since \(x_0\in\T^n\) and \(0<\tau\le T\) are arbitrary,
		\begin{equation}\label{eq:deg-weps-log}
			\|\partial_\ep u^\ep\|_{L^\infty(\T^n\times[0,T])}
			\le
			C\bigl(1+|\log\ep|\bigr).
		\end{equation}
		For \(0<\eta<\ep\),
		\[
		\|u^\ep-u^\eta\|_{L^\infty(\T^n\times[0,T])}
		\le
		\int_\eta^\ep
		C\bigl(1+|\log s|\bigr)\,ds
		\le
		C\ep\bigl(1+|\log\ep|\bigr).
		\]
		Letting \(\eta\to0\), and using the uniform convergence \(u^\eta\to u\), we
		obtain
		\[
		\|u^\ep-u\|_{L^\infty(\T^n\times[0,T])}
		\le
		C\ep\bigl(1+|\log\ep|\bigr).
		\]
		This proves Theorem~\ref{thm:intro-log-rate}.
	\end{proof}
	
	\begin{rem}
		Since the terminal datum of the adjoint equation is a Dirac mass,
		\(\sig^\ep\) may be singular near \(t=\tau\). For this reason, we split
		the time interval and introduce the cutoff function \(\chi_\tau(t)\) to
		avoid the endpoint singularity.
	\end{rem}
	
	\section{The averaged error estimate}
	\label{sec:averaged-estimate}
	
	In this section we prove Theorem~\ref{thm:intro-avg-rate}. We keep the notation
	from the previous section:
	\[
	A^\ep(x):=a(x)+\ep,
	\qquad
	b^\ep(x,t):=D_pH(x,Du^\ep(x,t)),
	\]
	and \(L^\ep\) is defined by \eqref{eq:deg-L-eps}. The pointwise estimate was
	obtained by using an adjoint density with terminal datum \(\delta_{x_0}\). We
	now replace this terminal datum by a smooth probability density. This removes
	the singular endpoint estimates which were responsible for the logarithmic
	loss.
	
	Let \(r\in C^\infty(\T^n)\), \(r\ge0\), and
	\[
	\int_{\T^n}r(x)\,dx=1.
	\]
	Let \(\sig^\ep\) solve
	\begin{equation}\label{eq:avg-adjoint}
		\begin{cases}
			-\sig^\ep_t-\Div(b^\ep\sig^\ep)-\Del(A^\ep\sig^\ep)=0,
			& (x,t)\in\T^n\times(0,T),\\[1mm]
			\sig^\ep(\cdot,T)=r,
			& x\in\T^n .
		\end{cases}
	\end{equation}
	Then
	\begin{equation}\label{eq:avg-sigma-basic}
		\sig^\ep\ge0,
		\qquad
		\int_{\T^n}\sig^\ep(x,t)\,dx=1
		\quad\text{for }0\le t\le T .
	\end{equation}
	As in the previous section, the adjoint relation gives, for every function to
	which the integrations by parts apply,
	\begin{equation}\label{eq:avg-duality}
		\int_{\T^n}f(x,T)r(x)\,dx
		-
		\int_{\T^n}f(x,0)\sig^\ep(x,0)\,dx
		=
		\int_0^T\int_{\T^n}L^\ep f\,\sig^\ep\,dx\,dt .
	\end{equation}
	
	Let
	\[
	w^\ep:=\partial_\ep u^\ep .
	\]
	Differentiating \eqref{eq:intro-eps} with respect to \(\ep\), we have
	\begin{equation}\label{eq:avg-weps-eq}
		L^\ep w^\ep=\Del u^\ep,
		\qquad
		w^\ep(x,0)=0 .
	\end{equation}
	Taking \(f=w^\ep\) in \eqref{eq:avg-duality}, we obtain
	\begin{equation}\label{eq:avg-adjoint-identity}
		\int_{\T^n}w^\ep(x,T)r(x)\,dx
		=
		\int_0^T\int_{\T^n}
		\Del u^\ep\,\sig^\ep\,dx\,dt .
	\end{equation}
	
	We first prove the entropy estimate for \(\sig^\ep\). Compared with
	Lemma~\ref{lem:deg-entropy}, no endpoint selection is needed, since the terminal
	datum is now the smooth density \(r\).
	
	\begin{lem}\label{lem:avg-entropy}
		Let \(\sig^\ep\) solve \eqref{eq:avg-adjoint}. Then, for every
		\(\del>0\),
		\begin{equation}\label{eq:avg-fisher}
			\begin{aligned}
				\int_0^T\int_{\T^n}
				A^\ep
				\frac{|D\sig^\ep|^2}{\sig^\ep}\,dx\,dt
				&\le
				\del
				\int_0^T\int_{\T^n}
				|D^2u^\ep|^2\sig^\ep\,dx\,dt                                      \\
				&\quad
				+
				C_\del
				\left(
				1+\int_{\T^n}r|\log r|\,dx
				\right).
			\end{aligned}
		\end{equation}
	\end{lem}
	
	\begin{proof}
		Set
		\[
		E(t):=\int_{\T^n}\sig^\ep(x,t)\log\sig^\ep(x,t)\,dx .
		\]
		The same computation as in Lemma~\ref{lem:deg-entropy} gives
		\begin{equation}\label{eq:avg-entropy-expanded}
			E'(t)
			=
			\int_{\T^n}
			A^\ep\frac{|D\sig^\ep|^2}{\sig^\ep}\,dx
			-
			\int_{\T^n}
			(\Div b^\ep+\Del a)\sig^\ep\,dx .
		\end{equation}
		Integrating from \(0\) to \(T\), we find
		\begin{equation}\label{eq:avg-fisher-identity}
			\begin{aligned}
				\int_0^T\int_{\T^n}
				A^\ep\frac{|D\sig^\ep|^2}{\sig^\ep}\,dx\,dt
				&=
				E(T)-E(0)
				+
				\int_0^T\int_{\T^n}
				(\Div b^\ep+\Del a)\sig^\ep\,dx\,dt .
			\end{aligned}
		\end{equation}
		Since \(\sig^\ep(\cdot,T)=r\),
		\[
		E(T)=\int_{\T^n}r\log r\,dx
		\le
		\int_{\T^n}r|\log r|\,dx .
		\]
		Moreover, \(E(0)\ge -C\), because \(s\log s\ge -e^{-1}\) for \(s\ge0\) and
		the torus has finite measure. 
		
		It remains to estimate the last term in \eqref{eq:avg-fisher-identity}.
		Recall that
		\[
		\Div b^\ep
		=
		H_{x_i p_i}(x,Du^\ep)
		+
		H_{p_i p_j}(x,Du^\ep)u^\ep_{x_i x_j}.
		\]
		By \eqref{eq:intro-lipschitz}, the \(C^2\)-regularity of \(H\) on the
		relevant compact set, the boundedness of \(\Del a\), and Young's inequality,
		we have, for every \(\del>0\),
		\begin{equation}\label{eq:avg-divb-bound}
			\left|
			\int_{\T^n}
			(\Div b^\ep+\Del a)\sig^\ep\,dx
			\right|
			\le
			\del
			\int_{\T^n}
			|D^2u^\ep|^2\sig^\ep\,dx
			+
			C_\del .
		\end{equation}
		Combining \eqref{eq:avg-fisher-identity} and \eqref{eq:avg-divb-bound}, we
		obtain \eqref{eq:avg-fisher}.
	\end{proof}
	
	\begin{rem}
		In the proof above, we repeatedly use the fact that the entropy
		\[
		\int_{\T^n}\sig^\ep\log\sig^\ep\,dx
		\]
		is bounded from below. This follows from the finiteness of \(\T^n\) and
		the normalization \(\int_{\T^n}\sig^\ep\,dx=1\). Such a lower bound is not
		automatic on \(\R^n\), and additional moment or confinement estimates would
		be needed in that setting.
	\end{rem}

	We next derive the averaged analogue of the weighted Hessian estimate in the
	previous section. Since the terminal datum is smooth, no time cut-off is needed.
	
	\begin{prop}\label{prop:avg-D2}
		Let \(\sig^\ep\) solve \eqref{eq:avg-adjoint}. Then
		\begin{equation}\label{eq:avg-D2}
			\int_0^T\int_{\T^n}
			|D^2u^\ep|^2\sig^\ep\,dx\,dt
			\le
			C
			\left(
			1+\|Dr\|_{L^1(\T^n)}
			+\int_{\T^n}r|\log r|\,dx
			\right).
		\end{equation}
	\end{prop}
	
	\begin{proof}
		Set
		\[
		\phi^\ep:=\Del u^\ep .
		\]
		From \eqref{eq:deg-bernstein-short}, we have
		\begin{equation}\label{eq:avg-bernstein}
			L^\ep\phi^\ep
			-
			2Da\cdot D\phi^\ep
			+
			\frac{\theta}{2}|D^2u^\ep|^2
			\le C .
		\end{equation}
		Pairing \eqref{eq:avg-bernstein} with \(\sig^\ep\) over \([0,T]\), and
		using \eqref{eq:avg-duality}, gives
		\begin{equation}\label{eq:avg-D2-start}
			\begin{aligned}
				\frac{\theta}{2}
				\int_0^T\int_{\T^n}
				|D^2u^\ep|^2\sig^\ep\,dx\,dt
				&\le
				C
				-
				\int_{\T^n}\phi^\ep(x,T)r(x)\,dx                                  \\
				&\quad
				+
				\int_{\T^n}\phi^\ep(x,0)\sig^\ep(x,0)\,dx                          \\
				&\quad
				+
				2\int_0^T\int_{\T^n}
				Da\cdot D\phi^\ep\,\sig^\ep\,dx\,dt .
			\end{aligned}
		\end{equation}
		
		We first estimate the two boundary terms. Since \(u^\ep(x,0)=g(x)\),
		\[
		\phi^\ep(x,0)=\Del g(x),
		\]
		and therefore
		\begin{equation}\label{eq:avg-initial-boundary}
			\left|
			\int_{\T^n}\phi^\ep(x,0)\sig^\ep(x,0)\,dx
			\right|
			\le
			\|\Del g\|_{L^\infty(\T^n)} .
		\end{equation}
		For the terminal term, integration by parts on \(\T^n\) yields
		\begin{equation}\label{eq:avg-terminal-boundary}
			\begin{aligned}
				\left|
				\int_{\T^n}\phi^\ep(x,T)r(x)\,dx
				\right|
				&=
				\left|
				-\int_{\T^n}Du^\ep(x,T)\cdot Dr(x)\,dx
				\right|                                                          \\
				&\le
				C\|Dr\|_{L^1(\T^n)} .
			\end{aligned}
		\end{equation}
		
		It remains to estimate the term involving \(Da\). Let
		\[
		J:=
		\int_0^T\int_{\T^n}
		Da\cdot D\phi^\ep\,\sig^\ep\,dx\,dt .
		\]
		Integrating by parts in space,
		\begin{equation}\label{eq:avg-Da-split}
			J
			=
			-\int_0^T\int_{\T^n}
			\phi^\ep(\Del a)\sig^\ep\,dx\,dt
			-
			\int_0^T\int_{\T^n}
			\phi^\ep Da\cdot D\sig^\ep\,dx\,dt .
		\end{equation}
		The first term on the right-hand side satisfies
		\begin{equation}\label{eq:avg-Da-first}
			\begin{aligned}
				\left|
				\int_0^T\int_{\T^n}
				\phi^\ep(\Del a)\sig^\ep\,dx\,dt
				\right|
				&\le
				C\int_0^T\int_{\T^n}
				|D^2u^\ep|\sig^\ep\,dx\,dt                                      \\
				&\le
				\del
				\int_0^T\int_{\T^n}
				|D^2u^\ep|^2\sig^\ep\,dx\,dt
				+
				C_\del .
			\end{aligned}
		\end{equation}
		For the second term, Lemma~\ref{lem:Da-structure} and \(a\le A^\ep\) imply
		\[
		|Da|^2\le C A^\ep .
		\]
		Thus, by Young's inequality,
		\begin{equation}\label{eq:avg-Da-second-start}
			\begin{aligned}
				\left|
				\int_0^T\int_{\T^n}
				\phi^\ep Da\cdot D\sig^\ep\,dx\,dt
				\right|
				&\le
				\del
				\int_0^T\int_{\T^n}
				|\phi^\ep|^2\sig^\ep\,dx\,dt                                     \\
				&\quad
				+
				C_\del
				\int_0^T\int_{\T^n}
				A^\ep\frac{|D\sig^\ep|^2}{\sig^\ep}\,dx\,dt .
			\end{aligned}
		\end{equation}
		Using \(|\phi^\ep|^2\le n|D^2u^\ep|^2\) and Lemma~\ref{lem:avg-entropy}, we
		obtain
		\begin{equation}\label{eq:avg-Da-second}
			\left|
			\int_0^T\int_{\T^n}
			\phi^\ep Da\cdot D\sig^\ep\,dx\,dt
			\right|
			\le
			\del
			\int_0^T\int_{\T^n}
			|D^2u^\ep|^2\sig^\ep\,dx\,dt
			+
			C_\del
			\left(
			1+\int_{\T^n}r|\log r|\,dx
			\right).
		\end{equation}
		Again, the small constants in Young's inequality and in Lemma~\ref{lem:avg-entropy}
		are chosen successively and then absorbed into the left-hand side. Combining
		\eqref{eq:avg-Da-split}, \eqref{eq:avg-Da-first}, and
		\eqref{eq:avg-Da-second}, we get
		\begin{equation}\label{eq:avg-Da-term-final}
			|J|
			\le
			\del
			\int_0^T\int_{\T^n}
			|D^2u^\ep|^2\sig^\ep\,dx\,dt
			+
			C_\del
			\left(
			1+\int_{\T^n}r|\log r|\,dx
			\right).
		\end{equation}
		
		Substituting \eqref{eq:avg-initial-boundary},
		\eqref{eq:avg-terminal-boundary}, and \eqref{eq:avg-Da-term-final} into
		\eqref{eq:avg-D2-start}, and choosing \(\del>0\) sufficiently small, we
		obtain \eqref{eq:avg-D2}.
	\end{proof}

	We now finish the proof of the averaged estimate.
	
	\begin{proof}[Proof of Theorem~\ref{thm:intro-avg-rate}]
		By \eqref{eq:avg-adjoint-identity},
		\begin{equation}\label{eq:avg-main-weps}
			\int_{\T^n}\partial_\ep u^\ep(x,T)r(x)\,dx
			=
			\int_0^T\int_{\T^n}
			\Del u^\ep\,\sig^\ep\,dx\,dt .
		\end{equation}
		Using \(|\Del u^\ep|\le \sqrt n |D^2u^\ep|\), Cauchy's inequality,
		\eqref{eq:avg-sigma-basic}, and Proposition~\ref{prop:avg-D2}, we obtain
		\[
		\begin{aligned}
			\left|
			\int_{\T^n}\partial_\ep u^\ep(x,T)r(x)\,dx
			\right|
			&\le
			C
			\int_0^T\int_{\T^n}
			|D^2u^\ep|\sig^\ep\,dx\,dt                                      \\
			&\le
			C
			\left(
			\int_0^T\int_{\T^n}
			|D^2u^\ep|^2\sig^\ep\,dx\,dt
			\right)^{1/2}
			\left(
			\int_0^T\int_{\T^n}\sig^\ep\,dx\,dt
			\right)^{1/2}                                                     \\
			&\le
			C
			\left(
			1+\|Dr\|_{L^1(\T^n)}
			+\int_{\T^n}r|\log r|\,dx
			\right)^{1/2}.
		\end{aligned}
		\]
		Applying the preceding estimate with \(\ep\) replaced by \(s\), and with the
		corresponding adjoint density associated with \(u^s\), we obtain, for
		\(0<\eta<\ep\),
		\[
		\begin{aligned}
			\left|
			\int_{\T^n}
			\bigl(u^\ep(x,T)-u^\eta(x,T)\bigr)r(x)\,dx
			\right|
			&\le
			\int_\eta^\ep
			\left|
			\int_{\T^n}\partial_su^s(x,T)r(x)\,dx
			\right|\,ds                                                        \\
			&\le
			C(\ep-\eta)
			\left(
			1+\|Dr\|_{L^1(\T^n)}
			+\int_{\T^n}r|\log r|\,dx
			\right)^{1/2}.
		\end{aligned}
		\]
		Letting \(\eta\to0\), and using the uniform convergence \(u^\eta\to u\), we
		get
		\[
		\left|
		\int_{\T^n}
		\bigl(u^\ep(x,T)-u(x,T)\bigr)r(x)\,dx
		\right|
		\le
		C\ep
		\left(
		1+\|Dr\|_{L^1(\T^n)}
		+\int_{\T^n}r|\log r|\,dx
		\right)^{1/2}.
		\]
		This proves Theorem~\ref{thm:intro-avg-rate}.
	\end{proof}
	
	\section*{Acknowledgments}
	
	The author is deeply grateful to Professor Hung V. Tran for suggesting this
	problem and for his continuous guidance and support throughout the development
	of this work. The author also thanks Professor Wei Cheng for his support during
	the author's undergraduate study and for the opportunity to present this work at
	the Tianyuan meeting.

\end{document}